\documentclass[11pt,a4paper]{article}
\usepackage[utf8]{inputenc}
\usepackage{geometry}
\usepackage[T1]{fontenc}
\usepackage{lmodern}
\usepackage{amsfonts}
\usepackage{amsmath}
\usepackage{amssymb}
\usepackage{amsthm}
\usepackage{mathtools}
\usepackage{fancyhdr}
\usepackage{enumitem}
\usepackage{mathrsfs}
\usepackage{hyperref}
\usepackage{cleveref}
\usepackage[all,cmtip]{xy}
\usepackage{manfnt}
\usepackage{marginnote}
\usepackage{tikz} 
\usepackage{tikz-cd}
\usepackage{graphicx}
\usepackage{mathabx}
\usepackage{tcolorbox}
\usepackage[mathscr]{euscript}
\usepackage{amsbsy}
\usepackage{float}

\usepackage[style = alphabetic]{biblatex}
\newtheoremstyle{mythm} 
    {.5em}                    
    {.5em}                    
    {\upshape}                   
    {}                           
    {\bfseries}                   
    {.}                          
    {.5em}                       
    {}  

\theoremstyle{mythm}
\newtheorem{thm}{Theorem}
\newtheorem{lm}[thm]{Lemma}

\newtheorem*{thmA}{Main Theorem}


\usetikzlibrary{decorations.markings}

\makeatletter
\tikzcdset{
  open/.code     = {\tikzcdset{hook, circled};},
  closed/.code   = {\tikzcdset{hook, slashed};},
  open'/.code    = {\tikzcdset{hook', circled};},
  closed'/.code  = {\tikzcdset{hook', slashed};},
  circled/.code  = {\tikzcdset{markwith = {\draw (0,0) circle (.375ex);}};},
  slashed/.code  = {\tikzcdset{markwith = {\draw[-] (-.4ex,-.4ex) -- (.4ex,.4ex);}};},
  markwith/.code ={
    \pgfutil@ifundefined%
    {tikz@library@decorations.markings@loaded}%
    {\pgfutil@packageerror{tikz-cd}{You need to say %
      \string\usetikzlibrary{decorations.markings} to use arrows with markings}{}}{}%
    \pgfkeysalso{/tikz/postaction = {
      /tikz/decorate,
      /tikz/decoration={markings, mark = at position 0.5 with {#1}}}
    }
  },
}
\makeatother

\makeindex

\newcommand{\R}{\mathbb{R}}

\renewcommand{\S}{\text{S}}

\DeclareMathOperator{\supp}{supp}
\let \H \relax
\DeclareMathOperator{\H}{H}
\DeclareMathOperator{\C}{C}

\tikzset{commutative diagrams/.cd,
comm/.style={start anchor=center,end anchor=center,draw=none}
}

\title{Cup product in bounded cohomology\\of negatively curved manifolds}
\author{Domenico Marasco}
\date{}

\begin{document}
\maketitle
\begin{abstract}
    Let $M$ be a negatively curved compact Riemannian manifold with (possibly empty) convex boundary. Every closed differential $2$-form $\xi\in\Omega^2(M)$ defines a bounded cocycle $c_\xi\in\C_b^2(M)$ by integrating $\xi$ over straightened $2$-simplices. In particular Barge and Ghys proved that, when $M$ is a closed hyperbolic surface, $\Omega^2(M)$ injects this way in $\H_b^2(M)$ as an infinite dimensional subspace. 
    We show that any class of the form $[c_\xi]$, where $\xi$ is an exact differential 2-form, belongs to the radical of the cup product on the graded algebra $\H_b^\bullet(M)$. 
\end{abstract}
\section{Introduction}
Bounded cohomology is a rich research field with various applications, but direct computation of bounded cohomology modules is a hard task. An important case is the free non-abelian group with $n\geq 2$ generators $F_n=F$. The bounded cohomology modules with real coefficients $\H_b^k(F)$ are infinite dimensional when $k=2$ or $k=3$, while it is still not known whether $\H_b^k(F)\neq0$ when $k\geq 4$. All the classes in $\H_b^2(F)$ can notoriously be represented as coboundaries of \textit{quasi-morphisms}.
There are various recent results investigating whether it is possible to construct a non-trivial bounded cocycle of degree $k\geq 4$ as the cup product of non-trivial quasi-morphisms; see \cite{BucherMonod+2018+1157+1162}, \cite{Heuer2020CupPI}, \cite{fournierfacio2020infinite} and \cite{amontova2021trivial}. All these results seem to suggest that $\cup\colon\H_b^2(F)\times\H_b^k(F)\to\H_b^{k+2}(F)$ could be trivial. In particular, in \cite{amontova2021trivial} the authors prove the following:
\begin{thm}
Let $\varphi$ be a $\Delta$-decomposable quasi-morphism and $\alpha\in\H_b^k(F)$, then 
$$[\delta^1\varphi]\cup\alpha=0\in\H_b^{k+2}(F).$$
\end{thm}

The main result of this paper has a similar flavour, but in a different context.
Let $M$ be a negatively curved compact Riemannian manifold with (possibly empty) convex boundary. Every differential $k$-form $\psi\in\Omega^k(M)$ defines a singular $k$-cochain $c_\psi\in\C^k(M)$ by integrating $\psi$ over straightened simplices. 
As we will see in Section 2.2, $c_\psi$ is bounded when $k\geq2$. Moreover, for every $\varphi\in\Omega^1(M)$ we have $\delta^1c_\varphi=c_{d\varphi}$, hence $c_\varphi$ is a \textit{quasi-cocycle} $c_\varphi$, i.e. a cochain with bounded differential. Degree one quasi-cocycles play in singular cohomology the very same role of quasi-morphisms in group cohomology. If we denote by $\text{E}\Omega^2(M)\subset\Omega^2(M)$ the space of exact forms we will show the following:
\begin{thmA}
Let $\xi\in\text{E}\Omega^2(M)$ and $\alpha\in\H_b^k(M)$, then
$$[c_{\xi}]\cup\alpha=0\in\H_b^{k+2}(M).$$
\end{thmA}
This is particularly interesting when $M=\Sigma$, a closed hyperbolic surface. In this case all the quasi-cocycles defined by non-trivial exact forms are non-trivial and thus $\text{E}\Omega^2(\Sigma)$ is an infinite dimensional subspace of $\H_b^2(\Sigma)$. This is true thanks to Theorem 3.2 of \cite{Barge1988}:
\begin{thm}
The map $\Omega^2(\Sigma)\to\H_b^2(\Sigma)$ that sends $\psi$ to $[c_\psi]$
is injective.
\end{thm}
\section*{Acknowledgements}
I want to thank Sofia Amantova, Francesco Fournier-Facio and Marco Moraschini for the useful dicussions and interesting comments. I also want to thank my Ph.D. supervisor Roberto Frigerio who has suggested the topic of this paper.
\section{Preliminaries}
\subsection{Bounded cohomology and differential forms}
Let $X$ be a topological space, we denote by $(\C^{\bullet}(X),\delta^{\bullet})$ its singular cochain complex with real coefficients and by $\H^{\bullet}(X)$ its singular cohomology with real coefficients.
Let $\S_k(X)=\{s\colon\Delta^k\to X\}$ be the set of $k$-singular simplices of $X$ and
define an $\ell^\infty$ norm on $\C^k(X)$ by setting, for every $\omega\in\C^k(X)$, $$\|\omega\|_\infty=\sup\left\{|\omega(s)|\;\big|\;s\in\S_k(X)\right\}.$$ 
The subspaces of bounded $k$-cochains $$\C^k_b(X)=\left\{\omega\in\C^k(X)\;\big|\;\|\omega\|_\infty<\infty\right\}$$
form a subcomplex $\C^{\bullet}_b(X)\subset\C^{\bullet}(X)$, whose homology will be denoted by $\H^{\bullet}_b(X)$. The $\ell^\infty$ norm descends to a seminorm on $\H^{\bullet}(X)$ and $\H_b^{\bullet}(X)$ by defining the seminorm of a class as the infimum of the norms of its representatives. The inclusion $\C^{\bullet}_b(X)\hookrightarrow\C^{\bullet}(X)$ induces a map
$$
c^{\bullet}\colon\H_b^{\bullet}(X)\to\H^{\bullet}(X)
$$
called the \emph{comparison map}. The kernel of $c^k$ is denoted by $\text{EH}_b^k(X)$ and called the \textit{exact bounded cohomology} of $X$.\\

Now let $X$ be a Riemannian manifold, we denote by $\Omega^k(X)$ the space of smooth $k$-forms on $X$ and by $d\colon\Omega^k(X)\to\Omega^{k+1}(X)$ the usual differential. The subspaces of closed and exact $k$-forms will be denoted by $\text{C}\Omega^k(X)$ and $\text{E}\Omega^k(X)$, respectively. We denote the De Rham cohomology of $X$ by $$\H_{dR}^\bullet(X)=\frac{\text{C}\Omega^\bullet(X)}{\text{E}\Omega^\bullet(X)}.$$
For every $\psi\in\Omega^k(X)$ and $x\in X$ set 
$$
\|\psi_x\|_\infty=
\sup\left\{|\psi_x(\underline{v})|\;\big|\;\underline{v}\in T_xX\text{ is a $k$-orthonormal frame}\right\}
$$
so that we can define an $\ell^\infty$ norm on $\Omega^k(X)$ as follows:
$$
\|\psi\|_\infty=\sup_{x\in X}\{\|\psi_x\|_\infty\}\in[0,+\infty].
$$
Of course, if $X$ is compact, then $\|\psi\|_\infty<\infty$ for every $\psi\in\Omega^\bullet(X)$.
Observe that for any $k$-dimensional immersed submanifold $D\hookrightarrow X$ we have that
$$
\left|\int_D\psi\right|\leq\int_D\|\psi\|_\infty d\text{Vol}=\text{Vol}_X(D)\cdot\|\psi\|_\infty.
$$
\subsection{Negatively curved manifolds and 2-forms}
Throughout the whole paper, let $M$ be a negatively curved orientable compact Riemannian manifold with (possibly empty) convex boundary. The universal covering $\widetilde{M}$ is continuously uniquely geodesic and thus for every $(x_0,\dots,x_k)\in\widetilde{M}^{k+1}$, by repeatedly coning on the $x_i$ one can define the straight $k$-simplex $[x_0,\dots,x_k]\in\S_k(\widetilde{M})$ as constructed in Section 8.4 of \cite{frigerio2017bounded}. 
The fundamental group $\pi_1(M)=\Gamma$ acts on the universal covering $\widetilde{M}$ via deck transformations and this defines in turn an action of $\Gamma$ on $C^{\bullet}_b(\widetilde{M})$. We denote by $\C_b^\bullet(\widetilde{M})^\Gamma$ the subcomplex of $\Gamma$-invariant cochains.
The covering map $p\colon\widetilde{M}\to M$ induces an isometric isomorphism of normed complexes $\C_b^{\bullet}(M)\xrightarrow[]{\cong}\C_b^{\bullet}(\widetilde{M})^{\Gamma}$.
Similarly, $\Gamma$ acts on $\Omega^k(\widetilde{M})$, we denote by $\Omega^k(\widetilde{M})^\Gamma$ the space of $\Gamma$-invariant $k$-forms of $\widetilde{M}$. By pulling-back via the covering projection we get the identification $\Omega^k(M)\xrightarrow[]{\cong}\Omega^k(\widetilde{M})^\Gamma$.\\

For any $\psi\in\Omega^k(\widetilde{M})^\Gamma$, we define a cochain $c_\psi\in\C^k(\widetilde{M})^{\Gamma}$ by setting for every $s\in\S_k(\widetilde{M})$,
$$
c_\psi(s)=\int_{[s(e_0),\dots,s(e_k)]}\psi.
$$
where $e_0,\dots,e_k$ are the vertices of the standard simplex $\Delta^k$.

Applying Stoke's Theorem we see that for every $s\in\S_{k+1}(\widetilde{M})$,
$$
(\delta^k c_\psi)(s)=c_\psi(\partial_{k+1}s)=\int_{\partial_{k+1}[s(e_0),\dots,s(e_{k+1})]}\psi=\int_{[s(e_0),\dots,s(e_{k+1})]}d\psi=c_{d\psi}(s)
$$
and thus mapping $\psi$ to $c_\psi$ defines a morphism of cochain complexes $I^{\bullet}\colon\Omega^{\bullet}(\widetilde{M})^\Gamma\to\C^{\bullet}(\widetilde{M})^\Gamma$. Furthermore, the fact that the \emph{straightening operator} $s\mapsto [s(e_0),\dots,s(e_k)]$ is $\Gamma$-equivariantly homotopic to the identity of $\C^k(\widetilde{M})$ (see e.g. \cite{frigerio2017bounded} Proposition 8.11) implies that the map induced by $I^\bullet$ on cohomology corresponds to the \emph{De Rham isomorphism} $\H_{dR}^{\bullet}(M)\xrightarrow{\cong}\H^\bullet(M)$ defined e.g. in Chapter 18 of \cite{lee2018introduction}.\\

Since the action of $\Gamma$ is cocompact we have $\|\psi\|_\infty<\infty$, for every $\psi\in\Omega^k(\widetilde{M})^\Gamma$. Furthermore, as shown in the second section of \cite{INOUE198283}, when $k\geq 2$ the volume of $[x_0,\dots,x_k]$ is bounded by a constant $V_k$ that depends only on $k$ and an upper bound of the curvature of $M$. This means that for every $s\in\S_k(\widetilde{M})$,
$$
|c_\psi(s)|=\left|\int_{[s(e_0),\dots,s(e_k)]}\psi\right|<V_k\cdot\|\psi\|_\infty
$$
and thus $c_\psi\in\C_b^k(\widetilde{M})^\Gamma$ is a bounded cochain.\\

We have a well defined map for $k\geq2$:
\begin{align*}
    I_b^k\colon\text{C}\Omega^k(M)&\to\H_b^k(M)\\
    \psi&\mapsto[c_\psi].
\end{align*}
Interestingly, since $I^\bullet$ induces the De Rham isomorphism we have the following commutative diagram:
\begin{center}
\begin{tikzcd}
\text{C}\Omega^k(M) \arrow[r, two heads] \arrow[d, "I_b^k"] & \H_{dR}^k(M) \arrow[d, "\cong"] \\
\H_b^k(M) \arrow[r, "c^k"]                                  & \H^k(M)                        
\end{tikzcd}
\end{center}
showing that the comparison map $c^k$ is surjective for $k\geq2$ (this is true in the much more general context of aspherical manifolds with Gromov hyperbolic fundamental group, see \cite{Mineyev2001}).

Furthermore, when $k>2$, for any $d\varphi\in\text{E}\Omega^k(M)$, 
$$I_b^k(d   \varphi)=[c_{d\phi}]=[\delta^{k-1}c_\varphi]=0\in\H_b^k(M),$$
meaning that the restriction $I_b^k\colon\text{E}\Omega^k(M)\to\H_b^k(M)$ is the zero map. This implies that $I_b^k$ descends on the quotient $\text{C}\Omega^k(M)/\text{E}\Omega^k(M)=\H_{dR}^k(M)$ to a map $\hat{I}_b^k\colon\H_{dR}^k(M)\to\H_b^k(M)$. We now have the following commutative diagram:
\begin{center}
\begin{tikzcd}
                           & \H_{dR}^k(M) \arrow[d, "\cong"] \arrow[ld, "\hat{I}_b^k"'] \\
\H_b^k(M) \arrow[r, "c^k"] & \H^k(M).                                                   
\end{tikzcd}
\end{center}
Therefore, up to the identification $\H_{dR}^k(M)\cong\H^k(M)$, for $k>2$ the map $\hat{I}_b^k$ provides a right inverse of the comparison map. On the one hand, this raises the interesting question of understanding the possible geometric properties of the elements in the image of $\hat{I}_b^k$; on the other hand, for $k>2$ differential forms produce only a finite dimensional subsbace of $\H_b^k(M)$.

On the contrary, in degree $2$, for every $\xi=d\varphi\in\text{E}\Omega^2(M)$, the primitive $c_\varphi\in\C^1(M)$ of $c_\xi$ is not necessarily bounded since the length of geodesic segments in $M$ is arbitrarily big and thus $[c_\xi]\in\text{EH}_b^2(M)$ may be non-trivial. In particular, when $M=\Sigma$, a closed hyperbolic surface, thanks to Theorem 2 $[c_\xi]$ is never trivial if $\xi\neq0$, and the space of exact forms $\text{E}\Omega^2(\Sigma)$ defines a infinite dimensional subspace of $\text{EH}_b^2(\Sigma)$:
\begin{center}
\begin{tikzcd}
\text{E}\Omega^2(\Sigma) \arrow[r, hook] \arrow[d, "I_b^2", hook] & \Omega^2(\Sigma) \arrow[r, two heads] \arrow[d, "I_b^2", hook] & \H_{dR}^2(\Sigma) \arrow[d, "\cong"] \\
\text{EH}_b^2(\Sigma) \arrow[r, hook]                             & \H_b^2(\Sigma) \arrow[r, two heads]                                    & \H^2(\Sigma).                        
\end{tikzcd}
\end{center}
\subsection{Smooth cohomology}
In this section we show that every class $\alpha\in\H_b^k(M)$ admits a representative that smoothly depends on the vertices of simplices. Moreover, in Lemma 3 we show an additional property of this representative that we will use in the next section.\\

Let $X$ be a topological space, we endow the set of singular $k$-simplices $\S_k(X)$ with the compact-open topology to define the subcomplex of the continuous cochains of $X$
$$
\C^k_{c}(X)=\{\omega\in\C^k(X)\;\big|\;\omega_{|\S_k(X)}\;\text{is continuous}\}.
$$
Moreover, we set $\C^k_{c,b}(X)=\C^k_{c}(X)\cap\C^k_b(X)$ and denote the homology of these complexes by $\H_c^{\bullet}(X)$ and $\H_{c,b}^{\bullet}(X)$, respectively.\\

Theorem 1.4 of \cite{Frigerio2011BoundedCC} states that if $X$ is path connected, paracompact and with contractible universal covering $\widetilde{X}$, then the inclusion of bounded continuous cochains in classical cochains
$$
i_{b}^{\bullet}\colon\C^{\bullet}_{c,b}(X)\to\C_{b}^{\bullet}(X)
$$
induces isometric isomorphisms on cohomology
\begin{align*}
i_{b}^{\bullet}&\colon\H^{\bullet}_{c,b}(X)\to\H_{b}^{\bullet}(X).
\end{align*}
Furthermore, there is an explicit formula for the inverse of these isomorphisms
\begin{align*}
\theta_{b}^{\bullet}=(i_{b}^{\bullet})^{-1}&\colon\H_{b}^{\bullet}(X)\to\H_{c,b}^{\bullet}(X).
\end{align*}

In what follows we will give the explicit formula of $\theta_b^\bullet$ in the case $X=M$. It is shown in Lemma 6.1 of \cite{Frigerio2011BoundedCC} that the isometric isomorphism $\C_{b}^{\bullet}(M)\cong\C_{b}^{\bullet}(\widetilde{M})^{\Gamma}$ induced by $p\colon\widetilde{M}\to M$ can be restricted to
$$
p_{c,b}^{\bullet}\colon\C_{c,b}^{\bullet}(M)\to\C_{c,b}^{\bullet}(\widetilde{M})^{\Gamma}.
$$
With the identifications $\C_{b}^{\bullet}(M)\cong\C_{b}^{\bullet}(\widetilde{M})^{\Gamma}$ and $\C_{c,b}^{\bullet}(M)\cong\C_{c,b}^{\bullet}(\widetilde{M})^{\Gamma}$ in mind, we will write out the explicit formula for the map
$$
\widetilde{\theta}_b^k\colon\C_b^k(\widetilde{M})^{\Gamma}\to\C_{b,c}^k(\widetilde{M})^{\Gamma}
$$
which induces the map $\theta_b^k$ on cohomology.
Since $M$ is compact, we can slightly modify the construction in Lemma 5.1 of \cite{Frigerio2011BoundedCC}, by
using a smooth partition of unity subordinate to a \textit{finite} open cover of $M$ and get a
smooth map $h_{\widetilde{M}}\colon\widetilde{M}\to[0,1]$ with the following properties:
\begin{itemize}
    \item[(i)] There is an $N\in\mathbb{N}$, such that for every $x\in\widetilde{M}$ there is a neighbourhood $W_x$ of $x$ such that the set $\{\gamma\in\Gamma\;|\;\gamma(W_x)\cap\supp (h_{\widetilde{M}})\}$ has at most $N$ elements.
    \item[(ii)]For every $x\in\widetilde{M}$, we have $\sum_{\gamma\in\Gamma} h_{\widetilde{M}}(\gamma x)=1.$
    \item[(iii)] $\supp h_{\widetilde{M}}$ is compact.
\end{itemize}
Let $\omega\in\C_b^k(\widetilde{M})^\Gamma$ and pick a basepoint $z\in\widetilde{M}$. We define the function $f_\omega\colon\widetilde{M}^{k+1}\to\mathbb{R}$ as follows:
$$
f_\omega(x_0,\dots,x_k)=\sum_{(\gamma_0,\dots,\gamma_k)\in(\Gamma)^{k+1}}
h_{\widetilde{M}}(\gamma_0^{-1}x_0)\cdot\ldots\cdot h_{\widetilde{M}}(\gamma_k^{-1}x_k)\cdot\omega([\gamma_0z,\dots,\gamma_kz]).
$$
Notice that the sum above is finite because of property (i). Finally we can define 
$$\widetilde{\theta}^k_{b}(\omega)(s)=f_\omega(s(e_0),\dots,s(e_k)).$$ Observe that $\widetilde{\theta}^k_{b}(\omega)$ is a $\Gamma$-invariant cocycle because $f_\omega$ is a $\Gamma$-invariant function, where $\Gamma$ acts on $\widetilde{M}^{k+1}$ diagonally.\\

In order to prove the Main Theorem we will need the following:
\begin{lm}
Let $\omega\in\C_b^k(\widetilde{M})^{\Gamma}$ and let $(x_1,\dots,x_k)\in(\widetilde{M})^k$.
Then the function
$$
f_\omega(-,x_1,\dots,x_k)\colon\widetilde{M}\to\R
$$
is smooth and the norm of its differential $df_\omega(-,x_1,\dots,x_k)\in\Omega^1(\widetilde{M})$ is bounded by a constant that does not depend on $(x_1,\dots,x_k)$.
\end{lm}
\begin{proof}
It is clear by construction that $f_\omega(-,x_1,\dots,x_k)$ is smooth. Moreover, expanding its differential
$$
df_\omega(-,x_1,\dots,x_k)=\sum_{(\gamma_0,\dots,\gamma_k)\in(\Gamma)^{k+1}}
dh_{\widetilde{M}}(\gamma_0^{-1}-)\cdot\ldots\cdot h_{\widetilde{M}}(\gamma_k^{-1}x_k)\cdot\omega([\gamma_0z,\dots,\gamma_kz])
$$
we see that, by property (i) of $h_{\widetilde{M}}$, there are at most $N^{k+1}$ non-zero summands and thus 
$$\|df_\omega(-,x_1,\dots,x_k)\|\leq N^{k+1}\cdot\|dh_{\widetilde{M}}\|_\infty\cdot\|\omega\|_\infty.$$
We can conclude since $\|\omega\|_\infty<\infty$ by assumption and $\|dh_{\widetilde{M}}\|_\infty<\infty$ because $h_{\widetilde{M}}$ has compact support.
\end{proof}
\section{Proof of the Main Theorem}
Let $\varphi\in\Omega^1(\widetilde{M})^\Gamma$ and $[\omega]\in\H_b^k(M)$, we look for a bounded primitive of $c_{d\varphi}\cup\omega\in\C_b^{k+2}(\widetilde{M})^\Gamma$.
Observe that $c_\varphi\cup\omega\in\C^{k+1}_b(\widetilde{M})^{\Gamma}$ is a (not necessarily bounded) primitive, in fact
$$
\delta^{k+1}(c_\varphi\cup\omega)=\delta^1(c_\varphi)\cup\omega=c_{d\varphi}\cup\omega.
$$
Of course, it is sufficient to find an $\eta\in\C^k(\widetilde{M})^\Gamma$ such that $c_\varphi\cup\omega+\delta^k\eta\in\C_b^{k+1}(\widetilde{M})^\Gamma$ is bounded.

We first replace $\omega$ with $\widetilde{\theta}^{k}_{b}(\omega)$, this can be done without loss of generality because as shown in the previous section the map $\widetilde{\theta}_b^\bullet$ induces an isomorphism on bounded cohomology.
Under this assumption we have that $\omega(s)=f_\omega(s(e_0),\dots,s(e_k))$ for every $s\in\S_k(\widetilde{M})$. Thus $(c_\varphi\cup\omega)(s)$ only depends on the vertices of $s\in\S_{k+1}(\widetilde{M})$, in fact
\begin{align*}
(c_\varphi\cup\omega)(s)&=c_\varphi([s(e_0),s(e_1)])\cdot\omega([s(e_1),\dots,s(e_{k+1})])\\
&=\int_{[s(e_0),s(e_1)]}\varphi\cdot f_\omega(s(e_1),\dots,s(e_{k+1})).
\end{align*}
Next, we define the function $\zeta\colon\widetilde{M}^{k+2}\to\mathbb{R}$ as follows
$$
\zeta(x_0,\dots,x_{k+1})=\int_{[x_0,x_1]}\varphi\cdot f_\omega(-,x_2,\dots,x_{k+1})
$$
where we see $f_\omega(-,x_2,\dots,x_{k+1})$ as a $0$-form (i.e. a smooth function). We observe that $\zeta$ is a $\Gamma$-invariant function (again using the diagonal action of $\Gamma$ on $\widetilde{M}^{k+2}$), in fact $\varphi$ and $f_\omega$ are $\Gamma$-invariant and for any $\gamma\in\Gamma$ we have that $[\gamma x_0,\gamma x_1]=\gamma[x_0,x_1]$.
\begin{lm}
For every $(x_0,\dots,x_{k+1})\in\widetilde{M}^{k+2}$ we have that
$$
(c_\varphi\cup\omega)([x_1,\dots,x_{k+1}])=\zeta(x_0,\dots,x_{k+1})-\sum_{i=2}^{k+1}(-1)^i\zeta(x_0,x_1,x_1,x_2,\dots,\hat{x_i},\dots,x_{k+1}).
$$
\end{lm}
\begin{proof}
Since $\omega$ is a cocycle we have that for any $z\in\widetilde{M}$,
\begin{align*}
0&=\delta^k\omega([z,x_1,\dots,x_{k+1}])\\
&=\omega([x_1,\dots,x_{k+1}])+\sum_{i=1}^{k+1}(-1)^i\omega([z,x_1,\dots,\hat{x_i},\dots,x_{k+1}])\\
&=f_\omega(x_1,\dots,x_{k+1})+\sum_{i=1}^{k+1}(-1)^if_\omega(z,x_1,\dots,\hat{x_i},\dots,x_{k+1})
\end{align*}
and thus 
$$
f_\omega(x_1,\dots,x_{k+1})=-\sum_{i=1}^{k+1}(-1)^if_\omega(-,x_1,\dots,\hat{x_i},\dots,x_{k+1}).
$$
We use this relation to conclude that
\begin{align*}
(c_\varphi\cup\omega)([x_1,\dots,x_{k+1}])=
&c_\varphi([x_0,x_1])\cdot\omega([x_1,\dots,x_{k+1}])\\
=&\int_{[x_0,x_1]}\varphi\cdot f_\omega(x_1,\dots,x_{k+1})\\
=&\int_{[x_0,x_1]}\varphi\cdot\left(-\sum_{i=1}^{k+1}(-1)^if_\omega(-,x_1,\dots,\hat{x_i},\dots,x_{k+1})\right)\\
=&-\sum_{i=1}^{k+1}(-1)^i\int_{[x_0,x_1]}\varphi\cdot f_\omega(-,x_1,\dots,\hat{x_i},\dots,x_{k+1})\\
=&\zeta(x_0,\dots,x_{k+1})-\sum_{i=2}^{k+1}(-1)^i\zeta(x_0,x_1,x_1,x_2,\dots,\hat{x_i},\dots,x_{k+1}).
\end{align*}
\end{proof}
We define $\eta\in\C^k(\widetilde{M})^\Gamma$ so that for every $s\in\S_k(\widetilde{M})$,
$$\eta(s)=\zeta(s(e_0),s(e_1),s(e_1),s(e_2),\dots,s(e_k)).$$
This cochain is $\Gamma$-invariant because the function $\zeta$ is.
As anticipated we will conclude by showing that $c_\varphi\cup\omega+\delta^k\eta$ is a bounded cochain. Since both $c_\varphi\cup\omega$ and $\delta^k\eta$ only depend on the vertices of simplices it will be enough to show that the function $(x_0,\dots,x_{k+1})\in\widetilde{M}^{k+2}\mapsto(c_\varphi\cup\omega+\delta^k\eta)([x_0,\dots,x_{k+1}])\in\mathbb{R}$ is bounded:
\begin{align*}
(c_\varphi\cup\omega+\delta^k\eta)([x_0,\dots,x_{k+1}])&=
\zeta(x_0,\dots,x_{k+1})\\
&-\sum_{i=2}^{k+1}(-1)^i\zeta(x_0,x_1,x_1,x_2,\dots,\hat{x_i},\dots,x_{k+1})\\
&+\sum_{i=0}^{k+1}(-1)^i\eta([x_0,\dots,\hat{x_i},\dots,x_{k+1}])\\
&=\zeta(x_0,x_1,x_2,\dots,x_{k+1})\\
&+\zeta(x_1,x_2,x_2,\dots,x_{k+1})\\
&-\zeta(x_0,x_2,x_2,\dots,x_{k+1}).
\end{align*}
Next we use Stoke's Theorem:
\begin{align*}
&\zeta(x_0,x_1,x_2,\dots,x_{k+1})\\
+&\zeta(x_1,x_2,x_2,\dots,x_{k+1})\\
-&\zeta(x_0,x_2,x_2,\dots,x_{k+1})\\
=&\int_{[x_0,x_1]}\varphi\cdot f_\omega(-,x_2,\dots,x_{k+1})\\
+&\int_{[x_1,x_2]}\varphi\cdot f_\omega(-,x_2,\dots,x_{k+1})\\
-&\int_{[x_0,x_2]}\varphi\cdot f_\omega(-,x_2,\dots,x_{k+1})\\
=&\int_{[x_0,x_1]\cup[x_1,x_2]\cup[x_2,x_0]}\varphi\cdot f_\omega(-,x_2,\dots,x_{k+1})\\
=&\int_{\partial[x_0,x_1,x_2]}\varphi\cdot f_\omega(-,x_2,\dots,x_{k+1})\\
=&\int_{[x_0,x_1,x_2]}d(\varphi\cdot f_\omega(-,x_2,\dots,x_{k+1})).
\end{align*}
The integration domain is a $2$-simplex with bounded area, this means that we only need to check that the norm of $d(\varphi\cdot f_\omega(-,x_2,\dots,x_{k+1}))\in\Omega^2(\widetilde{M})^\Gamma$ is bounded by a constant that does not depend on $(x_2,\dots,x_{k+1})$.
We expand
$$
d(\varphi\cdot f_\omega(-,x_2,\dots,x_{k+1}))=d\varphi\cdot f_\omega(-,x_2,\dots,x_{k+1})+
\varphi\wedge df_\omega(-,x_2,\dots,x_{k+1}).
$$
Both $\varphi$ and $d\varphi$ are $\Gamma$-invariant and since $\Gamma$ is cocompact $\|\varphi\|_\infty<\infty$ and $\|d\varphi\|_\infty<\infty$. The function $f_\omega(-,x_2,\dots,x_{k+1})$ is  bounded by $\|\omega\|_\infty$. Finally, as we saw in Lemma 3, $\|df_\omega(-,x_2,\dots,x_{k+1})\|_\infty$ is also bounded by a constant that does not depend on $(x_2,\dots,x_{k+1})$. This concludes the proof of our main Theorem.

\printbibliography
\end{document}